\pdfoutput=1
\documentclass[letter]{article}

\usepackage{nicefrac}
\usepackage[numbers]{natbib}
\usepackage{bbm}
\usepackage{comment}
\usepackage[all]{xy}\usepackage[latin1]{inputenc}        
\usepackage[dvips]{graphics,graphicx}
\usepackage{amsfonts,amssymb,amsmath,xcolor,mathrsfs, amstext}
\usepackage{amsbsy, amsopn, amscd, amsxtra, amsthm,authblk, enumerate,dsfont}
\usepackage{upref}
\usepackage{geometry}
\usepackage{lineno}
\usepackage{float}

\usepackage[colorlinks,
            linkcolor=blue,
            anchorcolor=green,
            citecolor=blue
            ]{hyperref}

\numberwithin{equation}{section}

\newcommand{\wt}{\widetilde}
\def\alb#1\ale{\begin{align*}#1\end{align*}}

\newcommand{\etal}{{\em et al.}\ }

\newcommand{\R}{\mathbb{R}}

\newcommand{\PXY}{\mathcal{P}(\X \times \Y)}

\newcommand{\clPXY}{\overline{\mathcal{P}(\X \times \Y)}}

\newcommand{\clPO}{\overline{\mathcal{P}(\Omega)}}
\newcommand{\X}{\mathcal{X}}
\newcommand{\Y}{\mathcal{Y}}
\newcommand{\XY}{\mathcal{X} \times \mathcal{Y}}

\newcommand{\conv}{\text{conv}}

\newcommand{\KL}{\rm{KL}}
\newcommand{\RKL}{\rm{RKL}}

\newcommand{\PES}{\pi_{\textrm{ES}}}
\newcommand{\MES}{\mu_{\textrm{ES}}}
\newcommand{\NES}{\nu_{\textrm{ES}}}

\newcommand{\PT}{\cal{P}(\cal{T})}
\newcommand{\T}{\cal{T}}

\newtheorem{theorem}{Theorem}[section]
\newtheorem*{theorem*}{Theorem}
\newtheorem{lemma}[theorem]{Lemma}
\newtheorem*{lemma*}{Lemma}
\newtheorem*{definition*}{Definition}
\newtheorem{proposition}[theorem]{Proposition}
\newtheorem*{proposition*}{Proposition}

\newtheorem*{corollary*}{Corollary}
\newtheorem{definition}[theorem]{Definition}
\newtheorem*{definitions*}{Definitions}

\newtheorem{example}{\bf Example}[section]
\newtheorem*{example*}{\bf Example}
\theoremstyle{remark}
\newtheorem{remark}{\bf Remark}[section]

\numberwithin{equation}{section}

\title{Geometry and duality of alternating Markov chains}
\author{Deven Mithal and Lorenzo Orecchia\\ University of Chicago}

\date{}
\begin{document}

\maketitle

\begin{abstract}
In this note, we realize the half-steps of a general class of Markov chains as alternating projections with respect to the reverse Kullback-Leibler divergence between convex sets of joint probability distributions. Using this characterization, we provide a geometric proof of an information-theoretic duality between the Markov chains defined by the even and odd half-steps of the alternating projection scheme.
\end{abstract}

\newpage

\tableofcontents

\newpage
\section{Introduction}
The object of study of this paper is the {\it alternating Markov chain}, which we define to be an irreducible, aperiodic Markov chain over a finite state space $\X$,  whose evolution is given as
$$
\pi_{t+1}(x, y) = \begin{cases}
 P[x|y] \cdot (\pi_{t})_{\Y}(y), & t \in 2 \cdot \mathbb{N}\\
 P[y|x] \cdot (\pi_{t})_{\X}(x), & t \in  2 \cdot \mathbb{N} + 1
\end{cases}
$$
with fixed conditional distributions $P[y|x],P[x|y]$ over a finite product space $\XY$. The resulting process on $\XY$ is a time-inhomogeneous Markov chain, termed the {\it product chain}, in which each coordinate is alternatively resampled according to the conditionals.
We call the original Markov chain defined on $\X$  by the two-step resampling procedure the {\it primal chain}, and that induced on $\Y$ the {\it dual chain}. We refer to either coordinate resampling operation as a {\it half-step} and
and consider the half-step evolution $(\pi_t \in \clPXY)_{t \geq 0}$ of the product Markov chain on $\clPXY$, the space of probability distributions over the product space $\XY$ with not necessarily full support. Additionally, we denote by $\PXY$ the space of probability distributions over $\XY$ with full support. 

A well-known example of this construction is the Swendsen-Wang dynamics for the Potts model \cite{grimmett2006random}, where the product space consists of the product of Potts-model assignments and subsets of edges in a percolation model, and the conditional distributions are dictated by the so-called Edwards-Sokal coupling. Generalizing this example, we assume that the input conditional distributions are compatible in the sense that there exists a measure $\pi \in \clPXY$
which satisfies the disintegrations
\begin{equation}\label{eq-disint-generic-x}
    \pi(A \times B) = \int_{B}P[A| x] \pi_{\X}(dx)
\end{equation}
\begin{equation}\label{eq-disint-generic-y}
    \pi(A \times B) = \int_{A}P[B| y] \pi_{\Y}(dy)
\end{equation}
where $\pi_\X$ and $\pi_\Y$ indicate the marginals of $\pi$ with respect to $\X$ and $\Y$ respectively. From the ergodicity of the primal chain, it is readily verified that such a coupling $\pi$  must be the unique stationary distribution for the product Markov chain and possess the stationary distribution on the primal (resp. dual) chain as its $\X$ (resp. $\Y$) marginal, and we make the following definition.

\begin{definition}[Generalised Edwards-Sokal Coupling]\label{def-general-edwards-sokal}
    Given an alternating Markov chain, as described above, the unique probability measure $\pi \in \clPXY$ which satisfies the disintegrations in Equations~\ref{eq-disint-generic-x}-\ref{eq-disint-generic-y} is called the generalized Edwards-Sokal (ES) measure, which we denote as $\PES \in \clPXY$.
\end{definition}
Note that by the ergodicity assumption, $(\pi_{ES})_{\X}$ has full support over $\X$, and we can assume that $(\pi_{ES})_{\Y}$ also has full support over $\Y$ by pruning $\XY$ as needed.

\begin{remark}[Factorizable Markov chains]
    The proposed construction of alternating Markov chains is closely related to the notion of a \textit{factorizable} Markov chain by Caputo~\etal~\cite{caputo2024entropycontractionsmarkovchains}. In particular, it can be shown that a Markov chain $M$ is a primal Markov chain in an alternating Markov chain if and only if $M$ is factorizable in the sense of \cite{caputo2024entropycontractionsmarkovchains} and ergodic, modulo a measure zero modification to the $x|y$-conditional.
\end{remark}

 Caputo \etal~\cite{caputo2024entropycontractionsmarkovchains} show that all such primal chains are reversible, but otherwise few general results are known about the properties of alternating Markov chains, particularly as it regards their convergence.

In this paper, we develop a novel geometric understanding of alternating Markov chains, based on alternating projections, which we use to give an alternative proof for an information-theoretic relationship between the entropy decay of the primal and dual chains. In particular, we study the induced geometry on the space of log-likelihoods, and relate the Kullback-Leibler divergence to a Bregman divergence which naturally represents this induced geometry, hereby compensating for the failure of the reverse Kullback-Leibler divergence to be Bregman.
 
 This duality property is an easy consequence of the chain rule for the $\KL$-divergence, but the geometric approach we present in this paper draws a close analogy between alternating Markov chains and the Sinkhorn algorithm from entropically regularized optimal transport, which itself is realized in the primal formulation as an alternating projection scheme with respect to the $\KL$-divergence. We believe this reinforces the ubiquity of the phenomenon of alternating projections with divergences of the space of probability distributions, and contributes to the overall story of divergence geometry over probability measures.

 \begin{remark}[Information geometry proof]
     It should be noted that the proof of duality in this paper, executed using Bregman divergence machinery, can be done using information geometry tools. This involves studying the appropriate convex sets as submanifolds of the probability simplex, but ultimately culminates in the application of an analogous Pythagorean theorem, and shares a similar geometric interpretation.
 \end{remark}

 While a comprehensive review of related literature is beyond the scope of this paper, we give some idea of what the ambient field looks like.

 The study of the rate at which a Markov chain converges in $\KL$-divergence, which is motivated by Pinsker's inequality, is an extensive field which has powerful implications for the mixing time and related quantities regarding convergence, and a reference for the general theory of entropy decay of Markov chains can be found in \cite{caputo2023lecture}.

 The idea that many interesting Markov chains somehow admit as their natural unit of evolution an intermediary half-step and that the state of the chain at these half-steps can prove to have powerful consequences of the chain viewed at its typical full-step granularity is an established one. Referenced above, ~\cite{caputo2024entropycontractionsmarkovchains} reconciles several notions of entropy decay at half-step and full-step levels, as well as the analogues of these notions in the continuum.

 Our motivating example of the Swendsen-Wang dynamics was initially introduced in \cite{SwendsenWang87}, and has been the subject of considerable inquiry, with mixing bounds typically fall into one of two camps: comparison of the Swendsen-Wang dynamics with the Glauber dynamics or some other local dynamic, or an intrinsic study of the entropy decay properties of the Swendsen-Wang dynamics. The former line of work essentially begins with \cite{ullrich2012rapid}, \cite{ullrich2013comparison}, \cite{Ullrich_2014} and \cite{Ullrich_2014_Thesis}, with contemporary examples of this work including \cite{gheissari2024spatial} which studies the closely related FK dynamics. However, due to the technical compromise of obscuring the Swendsen-Wang dynamic itself, such bounds are often loose. The intrinsic entropy decay literature includes \cite{blanca2021entropydecayswendsenwangdynamics}, \cite{blanca2021swendsenwangdynamicstrees}, \cite{blanca2023rapidmixingglobalmarkov} and \cite{blanca2023tractabilitysamplingpottsmodel}. From this perspective, Theorem~\ref{thm-duality} is of interest as it enables one to transfer entropy decay properties from the primal chain to the dual chain (see Remark~\ref{remark-entropy-transfer}), and this technique is illustrated by Section 8 of \cite{blanca2021entropydecayswendsenwangdynamics} as well as a similar result in \cite{ullrich2012rapid}.

There is also a literature connecting the two component Gibbs sampler/Glauber dynamics with alternating projection methods, but this work is rather different in flavor and implication(\cite{greenwood1998information}, \cite{diaconis2010stochastic}, \cite{qian2024gibbs}).

The Sinkhorn algorithm is a procedure from the theory of Entropically Regularized Optimal Transport~\cite{nutz2021introduction}, and its alternating projection structure immediately admits an analogous analysis to that conducted in this paper by virtue of the $\KL$-divergence being Bregman~\cite{cesa2006prediction}. A recent paper \cite{deb2023wasserstein} contains the definition of a Markov chain which follows an alternating conditional sampling procedure similar to that which we consider in this paper, but with the marginals being defined by the joint distribution in the dynamic at each step rather than a fixed joint measure which uniquely realizes those conditionals.

The study of the geometry of spaces of probability measures equipped a divergence as the notion of distance is rich and has foundations in the work of Csisz\'ar, some of which can be found in \cite{csiszar1967information}, \cite{csiszar1975divergence}, \cite{csiszar1984information} and \cite{csiszar2004information}. 

Lastly, an offshoot of the general field of information theory which represents the above geometric perspective is information geometry, which broadly studies the differential structure of manifolds consisting of probability distributions. Textbook treatments are given in \cite{amari2000methods} and \cite{amari2016information}, with \cite{Nielsen_2020} providing a modern, elementary survey, and \cite{hino2024geometry} providing an overview of the application of information geometry to the realization of the EM algorithm as a sequence of alternating projections.

\subsection{Alternating Markov chains}
Before we can state our main results, we need to briefly introduce some fundamental notions regarding the support of alternating Markov chains.

\begin{definition}
    For an alternating Markov chain, we let $\cal{T} \subseteq \XY$ equal the support of $\PES.$
\end{definition}

Because we will measure the distance from stationarity of the distribution $\pi_t$ by the KL-divergence $D_{\KL}(\pi, \PES)$, it will be necessary to guarantee  that $\pi_t$ and $\PES$ are mutually absolutely continuous for large enough $t$. To this end, we introduce the following definition.

\begin{definition}\label{def-burnin}
We let the {\it burn-in time} $t_0$ of the product chain be defined as the first half-step \textit{after} which $\pi_{t_0}$ and $\PES$ are mutually absolutely continuous, i.e.\ $\operatorname{supp}(\pi_{t_0}) = \cal{T},$ for all choices of initial distribution $\pi_0$.
\end{definition}
The following lemma is a simple consequence of the ergodicity of the primal chain.
\begin{lemma}\label{lem-finite-burn-in}
The burn-in time of the product chain in an alternating Markov chain is finite. For all $t \geq t_0,$ we have $\operatorname{supp}(\pi_t) = \cal{T}$. 
\end{lemma} 
A common setting arises when the primal Markov chain has strictly positive probability transition among all pairs of states. In this case, the burn-in time is $3$ half-steps. In general, when studying the asymptotics of the mixing time of a family of Markov chains in some system size $n$, one must be wary of any the dependence of the burn-in on $n$.

The definition of burn-in time
allows for $t \geq t_0$ to identify $\pi_t$ as a probability distribution over $\cal{T}$ with full support, i.e.\ an element of $\PT$, and so for the rest of the paper we will only consider $t \geq t_0$ and restrict our attention $\PT$.

\subsection{Results and applications}
The core observation of this paper is to realize the sequence of half-steps of the alternating Markov chain as a projection scheme in the space of probability measures $\PT$ onto the intersection of appropriate sets with respect to a divergence, i.e.\ as an instantiation of von Neumann's alternating projection method. Surprisingly, the correct divergence to use is the reverse of the standard $\KL$-divergence, i.e.\ for all
 $p,q \in \PT$ we have
$$
D_{\RKL}(p,q) := D_{\KL} (q,p)
$$
It is taken to be $+\infty$ if $q$ fails to be absolutely continuous with respect to $p.$
The appropriate sets over which this $\RKL$ projection takes places are exactly the sets of distributions with one matching conditional:
\begin{align*}
    S_1 &:= \{\pi \in \PT | \, \pi \text{ satisfies Equation \ref{eq-disint-generic-x}} \} \\
    S_2 &:= \{\pi \in \PT | \, \pi \text{ satisfies Equation \ref{eq-disint-generic-y}} \}
\end{align*}
for which $\{ \PES \} = S_1 \cap S_2$. 
The following fact is a simple exercise in measure theory, so we omit it.
\begin{lemma}\label{lem-convexity}
 The sets $S_1$ and $S_2$ are convex.   
\end{lemma}
To state our main structural lemma, we need to introduce a new definition.

\begin{definition}[Denormalization/log-denormalization (modification of \cite{ohara2017doublyautoparallelstructureprobability}, Definition 2)]\label{def-denormalization}
     Given a convex subset $S$ of $\PT$, we define the set $\wt S$ in $\R_{+}^{\T}$ by
     \begin{align*}
         \wt S := \{ \tau \cdot p| \, \tau \in \R_+, \, p \in S\}
     \end{align*}
     which we refer to as the denormalization of $S$. We then define the image of the set $\wt S$ under the coordinate-wise logarithm as the logarithmic denormalization, and denote this by $\log \wt S$.
\end{definition}

We are now ready to state our structural result on the characterization of the half-steps of the product chain as alternating projections, which is proven in Section~\ref{sec-geometric-proof}.

\begin{theorem}\label{thm-alt-proj}
    The half-steps in the probability space $\PT$ of the alternating Markov chain $(\pi_{t})_{t \geq t_0}$ are given by alternating projections onto the sets $S_1$ and $S_2$ with respect to the $\RKL$-divergence. 
    Equivalently, their log-likelihoods $(\ell_{t})_{t \geq t_0}$ are given by alternating projections onto the sets $\log \tilde{S}_1$ and $\log \tilde{S}_2$.
    That is, the half-steps satisfy for $t \geq  t_0$:
    \begin{align}
        \pi_{t + 1} &= \arg \min_{\pi \in S_1} D_{\RKL}(\pi, \pi_{t})  & \textrm{  and  }  & & \ell_{t + 1} &= \arg \min_{\ell \in \log \tilde{S}_1} B_{H^*}(\ell, \ell_{t}) \;\textrm{ for } t \in 2\mathbb{N}\\
        \pi_{t + 1} &= \arg \min_{\pi \in S_2} D_{\RKL}(\pi, \pi_{t}) & \textrm{  and  }  & & \ell_{t + 1} &= \arg \min_{\ell \in \log \tilde{S}_2} B_{H^*}(\ell, \ell_{t}) \; \textrm{ for } t \in 2\mathbb{N}+1
    \end{align}
    where $B_{H*}$ represents some Bregman divergence.
\end{theorem}
 
The restriction based on the burn-in time is also necessary to ensure that there exists a distribution, namely $\pi_{ES}$, in $S_1$ and $S_2$ with respect to which the current $\pi_{t}$ is absolutely continuous. 

The proof of Theorem~\ref{thm-alt-proj} demonstrates an elementary technique for accessing a Pythagorean relation for the $\RKL$ divergence which, since the $\RKL$ divergence is not Bregman, is typically only available using heavier Information Geometry machinery.
   
Next, we leverage the properties of Bregman divergences (see Section~\ref{sec-geometric-proof}) to exploit the variational characterization of Theorem~\ref{thm-alt-proj} to prove the aforementioned duality theorem, i.e.\ in the behavior of the divergences $D_{\RKL}(\PES, \pi_{t})$ for even and odd $t.$
For convenience, we explicitly introduce notation for the primal and dual marginals of the product distribution $(\pi_k)$:
\begin{align}\label{def-marginals}
    \mu_{k} := (\pi_k)_{\X}, \, \nu_{k} := (\pi_k)_{\Y} ,\, \MES := (\PES)_{\X}, \, \NES := (\PES)_{\Y}
\end{align}

With this setup, we are now in a position to state the duality theorem.

\begin{theorem}[Duality]\label{thm-duality}
    For $t \geq t_0$, it holds that:
     \begin{align}\label{eq-pythagorean-brief}
        D_{\RKL}(\PES, \pi_{t}) = D_{\RKL}(\PES, \pi_{t + 1}) + D_{\RKL}(\pi_{t + 1}, \pi_{t})
    \end{align}
    This implies that the sequence of divergences $(D_{\KL}(\pi_t, \pi_{ES}))_{t \geq t_0}$ is monotonically non-increasing. Additionally, the primal and dual Markov chains possess equivalent entropy decay in the sense that for all even $t$ greater than $t_0$,   
    $$
        D_{\RKL}(\mu_{ES}, \mu_{t}) \geq D_{\RKL}(\mu_{ES}, \mu_{t+1}) \geq 
        D_{\RKL}(\nu_{ES}, \nu_{t+1}) \geq
        D_{\RKL}(\nu_{ES}, \nu_{t+2})\geq 
        D_{\RKL}(\mu_{ES}, \mu_{t+2}). 
    $$
\end{theorem}
The first expression is a consequence of the chain rule for the $\KL$-divergence and yields the second by positivity, but the second expression can also be derived directly by the data processing inequality \cite{Polyanskiy_Wu_2024}. While succinct, the purely information theoretic justification sheds no light on the variational structure in Theorem~\ref{thm-alt-proj} underlying the Markov chain, nor does it suggest an analogy with the Sinkhorn algorithm.

The above theorem states that the quantity $D_{\RKL}(\pi_{k + 1}, \pi_{k})$ precisely constitutes the progress made by the alternating half-step dynamic in entropic convergence towards the generalised Edwards-Sokal coupling.

\begin{remark}[Entropy transfer]\label{remark-entropy-transfer}
    By Theorem~\ref{thm-duality}, beyond the burn-in time $t_0$, the entropy decay of the primal and dual Markov chains are out of phase by at most two half-steps, or one full Markov chain step. This allows one to transfer known statements about the entropy decay of the primal chain to the dual chain (see \cite{blanca2021entropydecayswendsenwangdynamics}, \cite{ullrich2012rapid}). 
\end{remark}

\section{The geometric proof}\label{sec-geometric-proof}
\subsection{Bregman divergences}

Bregman divergences represent the duality gap in the first-order characterization of their convex configuration function and therefore inherit favorable convex-analytic properties. In this subsection, we follow the presentation of Bregman divergences in Chapter 11 of \cite{cesa2006prediction}.
Considerations of alternating projection schemes with respect to Bregman divergences can be found in \cite{BREGMAN1967200} and \cite{bauschke2000dykstras}.

\begin{definition}[Legendre function (\cite{cesa2006prediction}, Section 11.2)]\    Given a nonempty subset $\mathcal{A} \subseteq \R^d$, a function $f: \mathcal{A} \to \R$ is said to be Legendre if $\mathcal{A}$ has convex interior, $f \in C^1(\textrm{int}(\mathcal{A}))$ and is strictly convex over the entirety of $\mathcal{A}$, and for a sequence $x_n \in \textrm{int}(\mathcal{A})$ with $x_n \to x \in \mathcal{A} \setminus \textrm{int}(\mathcal{A})$ we have that $||\nabla f(x_n) ||_2 \to \infty$. 
\end{definition}

\begin{definition}[Bregman divergence (\cite{cesa2006prediction}, Section 11.2)] \label{def-bregman-div}    
    Given a Legendre function $f$, the induced Bregman divergence $B_f: \mathcal{A} \times \textrm{int}(\mathcal{A}) \to \R$ is given by 
    \begin{equation}
        B_f(p, q) := f(p) - f(q) - \langle \nabla f(q), p - q \rangle
    \end{equation}
\end{definition}

Bregman divergences satisfy a number of useful properties.

\begin{lemma}[Properties of Bregman divergences (\cite{cesa2006prediction}, Section 11.2)] \label{lem-bregman-properties}
The following properties hold:
    \begin{enumerate}
        \item Non-negativity property: $B_f(p, q) \geq 0$ for all relevant $p, q$.
        \item Positivity property: $B_f(p, q) = 0 \iff p = q$.
        \item Reverse property: the Fenchel dual $f^*$ of a Legendre function $f$ is again Legendre, and the dual Bregman divergence $B_{f^*}$ satisfies $B_{f}(p, q) = B_{f^*}(\nabla f(q), \nabla f(p))$. 
    \end{enumerate}
\end{lemma} 

The Bregman projection of a point $q$ onto a convex set $S$ is the point $p$ in $S$ minimizing $B_f(p,q)$, which is uniquely defined by the strict convexity of $f.$ Bregman divergences involving Bregman projections admit a Pythagorean relation in analogy with the classical result for the $\KL$ divergence.

\begin{proposition}[Pythagorean Relation (\cite{cesa2006prediction}, Lemma 11.3 )]\label{prop-pythagorean-generic}
Given a Legendre function $f$, define the information/Bregman projection with respect to $B_f$ from $q \in \textrm{int}(\mathcal{A})$ onto $S \cap \mathcal{A} \neq \emptyset$ with $S$ closed and convex by
\begin{equation*}
    p' := \arg \min_{p \in S \cap \mathcal{A}} B_f(p, q)
\end{equation*}
We then have that, for all $p \in S,$
\begin{equation*}
    B_f(p, q) \geq B_f(p, p') + B_f(p', q)
\end{equation*}
with equality when $S$ is affine.
\end{proposition}

The Kullback-Leibler (KL) is the relevant example of a Bregman divergence for this paper, and the definition of the reverse $\KL$ divergence (RKL) follows immediately. Note that the $\RKL$ divergence is not itself Bregman \cite{jiao2014information}.
\begin{example}[KL and RKL divergences]\label{ex-kl}
    Choosing $f(x) = x \ln x$, the Bregman divergence $D_{\KL} := D_f$ is the KL divergence. One then defines the RKL divergence as $D_{\RKL}(p,q) := D_{\KL} (q,p)$. When $P\in \clPO$ is mutually absolutely continuous with respect to $Q \in \clPO$, they satisfy $D_f(P,Q) = D_g(Q,P)$.
\end{example}

The following decomposition of the $\KL$ divergence between probability measures on finite product spaces is an easy consequence of disintegration of measure.

\begin{proposition}[Chain rule for KL divergence \cite{Polyanskiy_Wu_2024}]\label{prop-chain-kl}
    Given $P, Q \in \clPXY$ with $P$ absolutely continuous with respect to $Q$, we have that 
    \begin{align*}
        D_{\KL}(P,Q) &= D_{\KL}(P_{\X},Q_{\X}) + D_{\KL}(P[y|x],Q[y|x]|P_{\X}) \\
        &:= D_{\KL}(P_{\X},Q_{\X}) + \int_{\X} D_{\KL}(P[y|x],Q[y|x])P_{\X}(dx)
    \end{align*}
\end{proposition}

\subsection{Entropy, denormalization and log-denormalization}
\label{sec-structure}

While it is possible prove our characterization in Theorem~\ref{thm-alt-proj} without appealing to the language of Bregman divergences, the application of Proposition~\ref{prop-pythagorean-generic} to obtain Theorem~\ref{thm-duality} requires that we work with a Bregman divergence.
The  fact that the reverse KL divergence is not a Bregman divergence over the simplex constitutes the main obstacle to our analysis.
We overcome it by re-parametrizing the space of measures over $\T$ in terms of their log-likelihoods, which we call the \textit{log-denormalization} of $\PT$, following terminology from information geometry~\cite{ohara2017doublyautoparallelstructureprobability}.

In this section, we will formally define this re-parametrization as originating from the choice of the entropy as our Bregman configuration function.
We conclude the section by proving our main structural lemma: the inverse images of the sets $S_1$ and $S_2$ under our re-parametrizations are affine sets, for which Proposition~\ref{prop-pythagorean-generic} applies with equality.

As is standard in online learning~\cite{cesa2006prediction} and information geometry~\cite{ohara2017doublyautoparallelstructureprobability} when working with the Bregman divergence of the entropy function, the requirement that the entropy $H$ be a Legendre function forces us to define it over the space of measures $\R_+^{\T}$ over $\T$ rather than the simplex $\PT$. In this context, the set $\R_+^{\T}$ is called the \textit{denormalization} of the open simplex $\PT:$
$$\forall \pi \in \R_+^{\T}, \qquad
H(\pi) := \sum_{\tau \in \T} \pi(\tau) \log \pi(\tau) - \pi(\tau) .
$$
For $\pi, \rho \in \R^{\T}_+$, the corresponding Bregman divergence takes the form:
$$
B_H(\pi, \rho) = \sum_{\tau \in \T} \pi(\tau) \log \frac{\pi(\tau)}{\rho(\tau)} - \pi(\tau) + \rho(\tau).
$$
The following proposition shows that this Bregman divergence can be decomposed into a KL-divergence and a non-negative term.
\begin{proposition}\label{prop-deform-KL}
Let $h:\R_+ \to \R$ be the scalar entropy function, 
i.e.\ $h(x) = x \log x -x.$ 
Then, for  all measures $\pi, \rho \in \R^{\T}_+$:
$$
B_H(\pi, \rho) = B_h\left(\pi(\T),\rho(\T))\right) + \pi(\T) \cdot D_{\KL}\left(\frac{\pi}{\pi(\T)}, \frac{\rho}{\rho(\T)}\right)
$$
\end{proposition}
\begin{proof}
    We have:
    \begin{align*}
        B_H(\pi, \rho) =&  \left(\sum_{\tau \in \T} \pi(\tau) \log \frac{\pi(\tau)}{\rho(\tau)}\right) - \pi(\T) + \rho(\T)\\
        = & \;\pi(\T) \cdot \left(  \sum_{\tau \in \T} \frac{\pi(\tau)}{\pi(\T)} \cdot \left(\log \frac{\nicefrac{\pi(\tau)}{\pi(\T)}}{\nicefrac{\rho(\tau)}{\rho(\T)}}\right)+ \log\frac{\pi(\T)}{\rho(\T)}\right) - \pi(\T) + \rho(\T)\\
        =& \; \pi(\T) \cdot \log \frac{\pi(\T)}{\rho(\T)} - \pi(\T) + \rho(\T) + \pi(\T) \cdot \left(  \sum_{\tau \in \T} \frac{\pi(\tau)}{\pi(\T)} \cdot \log \frac{\nicefrac{\pi(\tau)}{\pi(\T)}}{\nicefrac{\rho(\tau)}{\rho(\T)}}\right)\\
        = & \;B_h\left(\pi(\T),\rho(\T))\right) + \pi(\T) \cdot D_{\KL}\left(\frac{\pi}{\pi(\T)}, \frac{\rho}{\rho(\T)}\right)
    \end{align*}
\end{proof}

When $\pi$ and $\rho$ are normalized to belong to the simplex, i.e.\ $\pi, \rho \in \PT$, the first term in the previous equation cancels, to yield:
$$
B_H(\pi, \rho) = D_{\KL}(\pi,\rho)
$$

\subsection{Fenchel Duality and Log-Denormalization}
A key ingredient of our arguments will be a connection between the reverse KL divergence and the dual Bregman divergence $B_{H^*}$. For this reason, we need to introduce in more detail the latter divergence. To start, we notice that the gradient of the entropy $\nabla H$ is simply the coordinate-wise logarithm, 
while its inverse $\nabla H^*$ is the coordinate-wise exponential. We denote them as follows:
$$
\nabla H(\pi) = \log \pi \textrm{ and } \nabla H^*(\ell) = \exp \ell.
$$
This implies that the domain of the Fenchel dual $H^*$ is the total space $\nabla H (\R^{\T}_+) = \log \R^{\T}_+ = \R^{\T}$, which we will refer to as the \textit{log-denormalization} of the simplex $\PT.$ With this construction, Lemma~\ref{lem-bregman-properties} then implies that, for all $\pi, \rho \in \PT$:
$$
D_{\RKL}(\pi, \rho) = D_{\KL}(\rho, \pi) = B_{H}(\rho, \pi) = B_{H^*}(\log \pi, \log \rho).
$$
In other words, the reverse KL divergence can be seen as a Bregman divergence when we parametrize the space of probabilities in terms of log likelihoods. This is the main tool we will leverage in overcoming the technical obstacle described at the beginning of this section.

As we want to work in the log-denormalization, it will be useful to establish certain properties of the logarithmic denormalizations defined in Definition~\ref{def-denormalization}.

\begin{lemma}
\label{lem-doubly-autoparallel}
 The logarithmic denormalizations $\log \wt S_1$ and $\log \wt S_2$ of $S_1$ and $S_2$ defined in Definition~\ref{def-denormalization} are affine subspaces of $\R^{\T}.$
\end{lemma}
\begin{proof}
We carry out the proof for $S_1$ from which $S_2$ follows analogously.
Consider the following probability distributions $\{\mu_x\}_{x \in \X}$ over the simplex $\PT$:
  \begin{align*}
        \mu_x = \sum_{y\in \Y : (x,y) \in \T} P[y| x] \cdot e_{(x,y)},
    \end{align*}
where $e_{(x,y)}$ is the standard unit vector corresponding to the Dirac mass on $(x,y)$ in $\PT$.
This is the set of probability distributions with Dirac $\X$-marginal and appropriate $\Y$ conditional. We have:
    \begin{align}\label{eq-linear-rep}
        S_1 = \conv(\{\mu_x\}_{x \in \X}).
    \end{align}
    where $\conv$ denotes the convex hull within $\R^{\T}$. 
    It then follows that the denormalization has the form
    \begin{align*}
        \wt S_1 = \sum_{x \in \X} t_x \mu_x 
    \end{align*}
    with $t_x > 0$ for all $x \in \X$, i.e.\ it is the open cone generated by the distributions $\{\mu_x\}_{x \in \X}.$ Noting that each $\mu_x$ has disjoint nonzero sets of coordinates, the log denormalization has the form
    \begin{align*}
        \log \wt S_1 = \sum_{x \in \X} \log(t_x) \cdot \bigg( \sum_{y \in \Y} e_{(x,y)} \bigg) + 
        \sum_{(x,y) \in \T} \log P[y| x] \cdot e_{(x,y)}
    \end{align*}
    Letting $c_x = \log (t_x)$ for all $x \in \X$, we have $c \in \R^{\X}$ and as $c_x$ varies, this clearly describes an affine subset of $\R^{\T}$.
\end{proof}

\subsection{Geometry and duality}\label{sec-geo-duality}
As previewed in the previous section, we will make use of the log-denormalization, which originates from the gradient of the entropy $H$. With this in mind, we introduce a specific notation for the log-likelihood of the distribution of the product chain $\pi_t$. For all $t \geq t_0,$ we let
$$
\ell_t := \nabla H(\pi_t) = \log \pi_t.
$$
These exist because, after the burn-in time, we are guaranteed that $\pi_t$ has full support over $\T$. With this notation in mind, we are ready to prove Theorem~\ref{thm-alt-proj}.

\begin{proof}[Proof of Theorem~\ref{thm-alt-proj}]
By the definition of burn-in time $t_0$, we have that $\pi_t \in \PT$ and the log-denormalizations $\ell_t$ are well-defined.  
By the convexity of the sets $\log \wt S_1$ and $\log \wt S_2$ and the strict convexity of $B_H^*(\cdot, \ell_t)$, we immediately deduce that the minimizers in the problems above exist and are unique, justifying the 
$\arg \min$ notation.

We start proving the claim for the projection of the log-denormalization over $\log \wt S_1$ and $t$ even. The claim for $\log \wt S_2$ and $t$ odd follows in the same way.
For any $\ell \in \log \wt S_1$, we apply Lemma~\ref{prop-deform-KL} to obtain:
\begin{align*}
B_{H^*}(\ell, \ell_t) =  B_H(\pi_t, \exp \ell)
= B_h\left( 1 ,\sum_{\tau \in \T} \exp \ell (\tau)\right) + \pi_t({\T}) \cdot D_{\KL}\left(\pi_t, \frac{\exp \ell}{\sum_{\tau \in \T} \exp \ell(\tau)}\right)
\end{align*}
Because $\ell \in \log \wt S_1$, we have  that $\frac{\exp \ell}{\sum_{\tau \in \T} \exp \ell(\tau)} \in S_1$, so that it has $y|x$-conditionals equal to $P[y|x]$ for all $(x,y) \in \T.$
We use this fact when applying the chain rule for the $\KL$-divergence: 
\begin{align*}
B_{H^*}(\ell, \ell_t) =&  \,B_h\left( 1 ,\sum_{\tau \in \T} \exp \ell (\tau)\right)\\ + & \,\pi_t({\T}) \cdot D_{\KL}\left((\pi_t)_{\X}, \left(\frac{\exp \ell}{\sum_{\tau \in \T} \exp \ell(\tau)}\right)_{\X} \right)\\ + & \,\pi_t({\T}) \cdot \sum_{x \in \X} D_{\KL}(\pi_t(y|x), P[y|x]) \cdot \pi_t(x).
\end{align*}
Finally, consider minimizing this expression over the choice of $\ell \in \log \wt S_1$. The third term is fixed and independent of $\ell.$ The first and second term can be set to their minimum of $0$, by choosing:
$$
\sum_{\tau \in \T} \exp \ell(\tau) =1 \textrm{ and } \left(\frac{\exp \ell}{\sum_{\tau \in \T} \exp \ell(\tau)}\right)_{\X} = (\pi_t)_{\X}
$$
This means $\frac{\exp \ell}{\sum_{\tau \in \T} \exp \ell(\tau)}$ is the unique probability distribution, with $\X$-marginals equal to those $\pi_t$ and $y|x$-conditionals equal to $P[y|x]$. But this is exactly $\pi_{t+1}.$ Hence, it must be the case that the minimizer $\ell$ equals $\ell_{t+1}$ as required.

To complete the statement, we prove the corresponding statement regarding $\pi_{t+1}$ as the projection over $S_1.$ The statement for $S_2$ follows in the same way.
The proof is entirely analogous to the one for the log-denormalization. For any $\pi \in \PT$ with $\pi \in S_1,$ we have, by the chain rule for the \KL-divergence:
$$
D_{\RKL}(\pi, \pi_t) = D_{KL}(\pi_t, \pi) = D_{\KL}((\pi_t)_\X, (\pi)_{\X}) + \sum_{x \in \X} D_{\KL}(\pi_t(y|x), P[y|x]) \cdot \pi_t(x).
$$
When minimizing over $\pi \in S_1,$ we obtain that the $\pi$ must be the unique distribution with $\X$-marginals equal to $(\pi_t)_{\X}$ and $y|x$ conditionals equal to $P[y|x]$. But this is exactly the distribution $\pi_{t+1}$, as required.
\end{proof}

Having established Theorem~\ref{thm-alt-proj} and Lemma~\ref{lem-doubly-autoparallel}, we are in a position to apply Proposition~\ref{prop-pythagorean-generic} to bound the entropy decay of the primal and dual Markov chains. 

\begin{proof}[Proof of Theorem~\ref{thm-duality}]
Since the sets $\log \wt S_1, \log \wt S_2$ have affine logarithmic denormalization, it follows from Theorem~\ref{thm-alt-proj} and Proposition~\ref{prop-pythagorean-generic} that 
    \begin{align}
    D_{\RKL}(\PES, \pi_{t}) = B_H(\pi_t, \PES) =
        B_{H^*} (\log \PES, \ell_t) = B_{H^*} (\log \PES, \ell_{t+1}) + B_{H^*}(\ell_{t+1}, \ell_t) =\\
        D_{\RKL}(\PES, \pi_{t + 1}) + D_{\RKL}(\pi_{t + 1}, \pi_{t})
    \end{align}
    The monotonicity of $(D_{\KL}(\pi_t, \pi_{ES}))_{t \geq t_0}$ follows then from the non-negativity of the last term in the equation above. The second expression in Theorem~\ref{thm-duality} then follows from positivity of the Bregman divergence.
\end{proof}

\bibliographystyle{plain}
\bibliography{paper}

\appendix

\end{document}